\documentclass[11pt, psamsfonts]{amsart}

\usepackage[a4paper,margin=4cm]{geometry}
\usepackage{amsfonts,amssymb,amscd}
\usepackage{graphicx,mathrsfs} % para usar el comando \mathscr{}
\usepackage{subfigure}
\newtheorem{theorem}{Theorem}[section]

\newtheorem{lemma}[theorem]{Lemma}

\newtheorem{proposition}[theorem]{Proposition}

\newtheorem{definition}[theorem]{Definition}

\theoremstyle{remark}

\newtheorem{remark}[theorem]{Remark}
\newtheorem{example}[theorem]{Example}

\renewcommand{\leq}{\leqslant}
\renewcommand{\geq}{\geqslant}

\newcommand{\ptl}{\partial}

\newcommand{\ga}{\gamma}

\newcommand{\wtilde}{\widetilde}

\newcommand{\nn}{\mathbb{N}}

\DeclareMathOperator{\area}{area}
\DeclareMathOperator{\inte}{int}


\numberwithin{equation}{section}

\begin{document}

\title[Minimizing trisections for the maximum relative diameter]
{Trisections of a 3-rotationally symmetric planar convex body minimizing
the maximum relative diameter}

\author[A. Ca\~nete]{Antonio Ca\~nete}
\address{Departamento de Matem\'atica Aplicada I \\ Universidad de Sevilla}
\email{antonioc@us.es}

\author[C. Miori]{Cinzia Miori}
\address{Departamento de An\'alisis Matem\'atico \\ Universidad de Alicante}
\email{cm4@alu.ua.es}

\author[S. Segura Gomis]{Salvador Segura Gomis}
\address{Departamento de An\'alisis Matem\'atico \\ Universidad de Alicante}
\email{salvador.segura@ua.es}

\thanks{The first author is partially supported by MCyT research project MTM2010-21206-C02-01.
The second author is partially supported by MINECO (Ministerio de Econom\'{\i}a y Competitividad) and FEDER (Fondo Europeo de Desarrollo Regional) project MTM2012-34037 and Fundaci\'{o}n S\'{e}neca project 04540/GERM/06, Spain. This research is a result of the activity developed within the framework of the Programme in Support of Excellence Groups of the Regi\'{o}n de Murcia, Spain, by Fundaci\'{o}n S\'{e}neca, Regional Agency for Science and Technology (Regional Plan for Science and Technology 2007-2010).}
\date{\today}

\begin{abstract}
In this work we study the fencing problem consisting of finding a
trisection of a 3-rotationally symmetric planar convex body
which minimizes the maximum relative diameter. We prove that an optimal solution
is given by the so-called standard trisection.
We also determine the optimal set giving the minimum value for this functional
and study the corresponding universal lower bound.
% a lower bound for the maximum relative diameter for any trisection of a set of this family.
\end{abstract}

\maketitle

\section*{Introduction}

In the setting of convex geometry, it is considered historically that
the classical geometric magnitudes associated to a planar compact convex set are
the perimeter, the area, the diameter, the minimum width, the circumradius and the inradius.
These magnitudes have been deeply studied, and along the last century,
the different relations between them, which can be usually established by means of inequalities,
have become a very interesting topic, as well as the characterization of the optimal %(or extremal)
sets satisfying the equality in such inequalities.
Possibly, the most remarkable example of this kind of problems is the classical isoperimetric inequality,
relating the area and the perimeter of a set, see \cite{O}.
We refer the reader to \cite{sa} and references therein for a detailed description
of the known relations involving two and three of the previous magnitudes for general planar convex sets.

In addition, the study of these classical functionals has inspired a large amount
of related questions and open problems, see \cite{cfg}.
Among them, we have the so-called \emph{fencing problems},
which consist of finding the way of partitioning a planar compact convex set into $n$ connected subsets of equal areas,
minimizing (or maximizing) a given geometric measure.
For these questions, different approaches can be explored: one can try to characterize the optimal division,
or find out some geometric properties of the solution,
or even compute estimates for the optimal value or bounds of the considered measure.

The most common sample of fencing problems is the \emph{relative isoperimetric problem},
consisting of minimizing the relative perimeter of the partition of the set
(or equivalently, the length of the partitioning curves).
For some regular polygons, including the circle, this relative isoperimetric problem
has been treated in several works \cite{tomonaga,bleicher,cr,cox}, not only for subdivisions of equal areas,
providing general properties of the minimal partition and characterizing the solution when the number of subsets is small.
On the other hand, for partitions into two connected subsets of equal areas, which are called bisections,
a recent result gives a sharp upper bound for the perimeter of the shortest bisection for planar convex sets,
which is attained by the circle \cite[Th.~1.1]{efknt}. This optimal bound is expressed in terms of the area of the set.
Some others lower and upper bounds for the minimal length of bisections
in terms of the classical magnitudes, and the corresponding characterizations of some optimal sets, can be found in \cite{gkms}, as well as the study of the same questions regarding the maximal length of bisections.
%On the other hand, for the analogous fencing problem where the relative perimeter is maximized instead of minimized,
%only some estimates for the the maximal length of the segment dividing a planar convex set into two subsets of equal areas are known: lower bounds were obtained in \cite{radz} and \cite{eggl},
%in terms of the diameter and the minimum width, or the circumradius of the consider convex set.

Apart from the relative perimeter of the partition,
fencing problems can be considered with others relative geometric magnitudes to be minimized or maximized.
%, see \cite[Problem~A26]{cfg}. % en ese problema se indica que hay muchas variantes para el fencing problem, aunque no se escriba explicitamente que se puede considerar el problema con el funcional diametro relativo.
For instance, we can define the \emph{maximum relative diameter} as follows:
given a planar compact (convex) set $C$ and a partition $\mathcal{P}$ of $C$ into
connected subsets $\{C_1,\dots,C_n\}$ of equal areas, the maximum relative diameter of $\mathcal{P}$ is
\[
d_M(\mathcal{P})=\max\{D(C_i):i=1,\dots,n\},
\]
where $D(C_j)=\max\{d(x,y):x,y\in C_j\}$ denotes the diameter of the subset $C_j$.
In other words, this magnitude measures the largest distance in any of the subsets of the partition.
This functional and the corresponding associated fencing problem
have been already studied in \cite{mps} (see also \cite{cms}).
This work focuses on bisections of centrally symmetric planar convex bodies,
and proves that for any set of this family, the minimum value for $d_M$ is attained
when the bisection consists of a straight line passing through the center of the set \cite[Prop.~4]{mps}.
We remark that the precise characterization of the optimal straight line is not known,
and that the minimizing bisection is not necessarily unique.
Moreover, it is also determined the centrally symmetric planar convex body of unit area with
the minimum possible value for the maximum relative diameter $d_M$,
that is, the optimal set for this problem \cite[Th.~5]{mps}.

The results in \cite{mps} are obtained in the class of centrally symmetric convex bodies,
which a priori may seem a restrictive hypothesis.
However, this class of sets is the suitable domain in the setting of bisections,
arising naturally in this framework:
when considering simple bisections by line segments,
we have that the set attaining the minimum possible value
for the maximum relative diameter is necessarily a centrally symmetric body,
and not only in the planar case \cite[Th.~7]{mps}.
Moreover, this class of sets had previously appeared for
some classical geometric questions.
Regarding the maximal length of the shortest bisecting chord,
Santal\'o conjectured if for any planar convex set $K$, an upper bound was given by $(4\pi)^{1/2}\,\area(K)$,
with equality for a disk \cite[Problem~A26]{cfg}. Although this conjecture is not true in general
(counterexamples were obtained by Auerbach \cite{auer}, see also \cite{fp}) and the complete answer has been recently found \cite[Th.~1.2]{efknt},
the conjecture holds when restricting to centrally symmetric convex sets \cite[Th.~4]{ci}.

%or the longest least-perimeter general bisection among planar convex sets of fixed area (\cite{polya}, \cite{ci}, \cite{g} \cite{efknt}).

% en este parrafo es cuando aparece por primera vez la palabra body, y se recuerda al final del parrafo cual es su significado.

Continuing in the direction explored in \cite{mps}, in this paper we shall focus on \emph{trisections}
of planar convex bodies, studying the minimum value for the maximum relative diameter.
Trisections are partitions into three connected subsets of equal areas by means of three curves meeting
in a common interior point, and therefore, the convenient setting for this problem is the
class of 3-rotationally symmetric planar convex bodies.
For any set in this class, we shall characterize a trisection that minimizes the
maximum relative diameter $d_M$ (Theorem~\ref{th:main}), describing its explicit construction in Section~\ref{se:tb}.
This particular minimizing trisection will be referred to as \emph{standard trisection}.
We shall also see that uniqueness of the solution is not expected,
since some small perturbations of a minimizing trisection will give the same value for $d_M$.
In addition, we shall determine which is the 3-rotationally symmetric planar convex body
attaining the minimum value for the maximum relative diameter (Theorem~\ref{th:optimo}).
This optimal set consists of a certain intersection of an equilateral triangle and a circle,
and has already arisen in several optimization problems involving some triplets of classical geometric magnitudes, see~\cite[\S.~4]{hss}.
We stress that, along this paper, we are considering trisections by general curves, not only by straight line segments,
and that the sets of our family are convex bodies (that is, compact convex sets),
in order to guarantee the existence of the maximum relative diameter.

The restriction of our fencing problem to the family of 3-rotationally symmetric planar convex bodies
is natural in this setting since we are dealing with trisections.
In fact, this family will play the same role as the class of centrally symmetric bodies for bisections.
Apart from this, the family of 3-rotationally symmetric planar convex bodies
is geometrically interesting for different reasons:
many geometric transformations preserve the existing threefold symmetry,
and it contains the optimal solution for many problems in the convex geometry setting
(for instance, Pal proved that among all planar convex sets with fixed minimal width,
the equilateral triangle is the set with minimum area enclosed~\cite{pa}).
Moreover, many boundary curves in Santal\'o's diagrams for complete systems of inequalities represent 3-rotationally symmetric convex bodies~\cite{hss}.
%, and these sets provide interesting solutions for lattice problems or for packing and covering problems.
%(for instance, among all convex planar sets of fixed area, the minimum value for the maximum relative diameter for bisections by a straight line is attained by a centrally symmetric set \cite[Th.~6]{mps}).

% We remark that, apart from its own mathematical interest, there are some civil applications of this fencing problem, since some legislation bounding the diameters of parcels networks is established in order to have a control of the resulting distances \cite{mps}.

%Apart from its own mathematical interest, we remark that some civil applications of this fencing problem can be found:
%there are regulations on urban subdivisions in town planning stating minimal requirements for the diameter of the subsets of the corresponding structure \cite{mps}.

We have organized this paper as follows. Section \ref{se:preliminaries} contains the precise definitions and the
statement of the problem, and Section~\ref{se:tb} describes the construction of the so-called standard trisection for any
3-rotationally symmetric planar convex body. This construction will be completely determined by the
smallest equilateral triangle containing the considered set.
Section~\ref{se:general} contains the proof of our main Theorem~\ref{th:main}:
\begin{quotation}
\emph{For any 3-rotationally symmetric planar convex body,
a trisection minimizing the maximum relative diameter $d_M$
is given by the standard trisection.}
\end{quotation}
%We start considering 3-rotationally symmetric polygons.
%Notice that these polygons will have necessarily $3n$ edges, $n\in\nn$.
%Firstly, in Section \ref{se:regular} we prove that the standard trisection
%provides the minimum value for the maximum relative diameter for \emph{regular} 3-rotationally symmetric polygons.
%The analysis of the particular case of the regular hexagon suggests the idea for the demonstration in the
%general case, except for the equilateral triangle, which requires a different study.
%In Section \ref{se:non-regular}, the same result is proved for \emph{non-regular} 3-rotationally symmetric polygons,
%as well as a kind of characterization for these sets depending on the placement of some of the vertices.

\noindent The key results for proving our main Theorem~\ref{th:main} are Proposition~\ref{prop:borde},
which indicates the pairs of points that may realize the maximum relative diameter for the standard trisection,
%indicating which pair of points realizes the maximum relative diameter for the standard trisection,
and Lemmata~\ref{le:primero} and \ref{le:segundo}, where we show that the maximum relative diameter for any
trisection is always greater than or equal to such value, then yielding that the standard trisection provides a desired minimum.
We also give some examples illustrating that the uniqueness of the solution for this problem does not hold.

In Section~\ref{se:min} we investigate a related problem in this setting,
searching for the 3-rotationally symmetric planar convex body of fixed area
with the minimum value for the maximum relative perimeter
(equivalently, the minimum possible value for the maximum relative diameter in that class of sets).
In view of Theorem~\ref{th:main}, we only have to focus on the standard trisections,
and after analyzing this question in some particular cases, we obtain our Theorem~\ref{th:optimo}:
\begin{quotation}
\emph{The minimum value for the maximum relative perimeter in the class of 3-rotationally symmetric planar convex bodies
of unit area is uniquely attained by the standard trisection of the set $\wtilde{H}$ from Figure~\ref{fig:optimo}.}
\end{quotation}
\begin{figure}[ht]
   \includegraphics[width=0.23\textwidth]{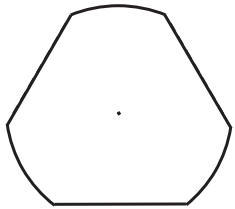}\\
  \caption{Optimal set providing the minimum value for the maximum relative diameter}
  \label{fig:optimo}
\end{figure}
This set $\wtilde{H}$ is geometrically determined by an appropriate intersection of an equilateral triangle and a circle,
and provides the universal sharp lower bound for the maximum relative diameter functional (Remark~\ref{re:quotient}).

Finally, in Section~\ref{se:final} we treat the same problem
when considering partitions into three connected subsets of equal areas which are not trisections.
Proposition~\ref{prop:general} asserts that, even in this more general case, the standard trisection
is a minimizing one for the maximum relative diameter.

We remark that Section~\ref{se:min} has been written constructively,
reproducing in some sense the process followed in order to obtain the results,
with the analysis of some particular examples leading to more general conclusions.
%showing the followed process in order to obtain the main results.
%following the way that has led to the main results.
%This can be checked especially in Section~\ref{se:min}.
We think this can be interesting for the reader, since it may help to a better comprehension of this work.
%and showing the methods used along this paper.

% Comparison of magnitudes=inequalitties; caracterizacion=sistema completo de desigualdades

\section{Preliminaries}
\label{se:preliminaries}

Given a planar convex set $C$, and $n\in\nn$, we say that $C$ has \emph{$n$-fold rotational symmetry}
(or that it is \emph{$n$-rotationally symmetric}) with respect to a center $p\in C$
if $C$ is invariant under the action of rotation of angle $\displaystyle{360/n}$ degrees about $p$.
In particular, for $n=3$ this implies that 3-rotationally symmetric sets
are invariant under any term of rotations about $p$ with angles of $120$, $240$ and $360$ degrees.
Consequently, any three equiangular axes centered at $p$ will divide $C$ into three identical subsets, up to rotations.
Some examples of sets of this family are the regular polygons of $3n$ edges (with $n\in\nn$) and
Reuleaux triangle.

%We shall say that a planar convex set $C$ is \emph{3-rotationally symmetric} if there exists a \emph{center} $p\in C$ such that $C$ is invariant under any tern of rotations centered at $p$, with angles of $2\pi/3$, $4\pi/3$ and $2\pi$ radians.

\begin{figure}[h]
\centering{
\subfigure{\includegraphics[width=0.16\textwidth]{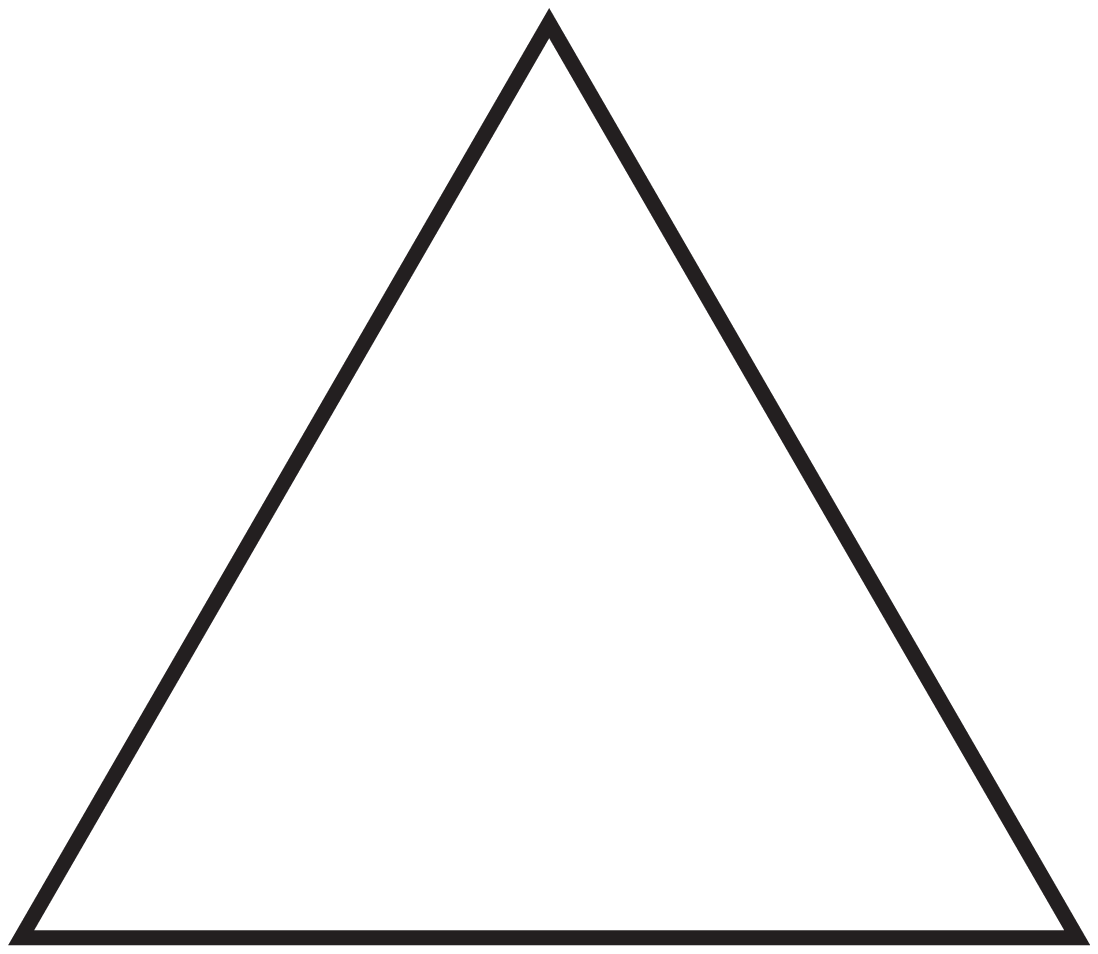}}
\hspace{0.07\textwidth}
\subfigure{\includegraphics[width=0.16\textwidth]{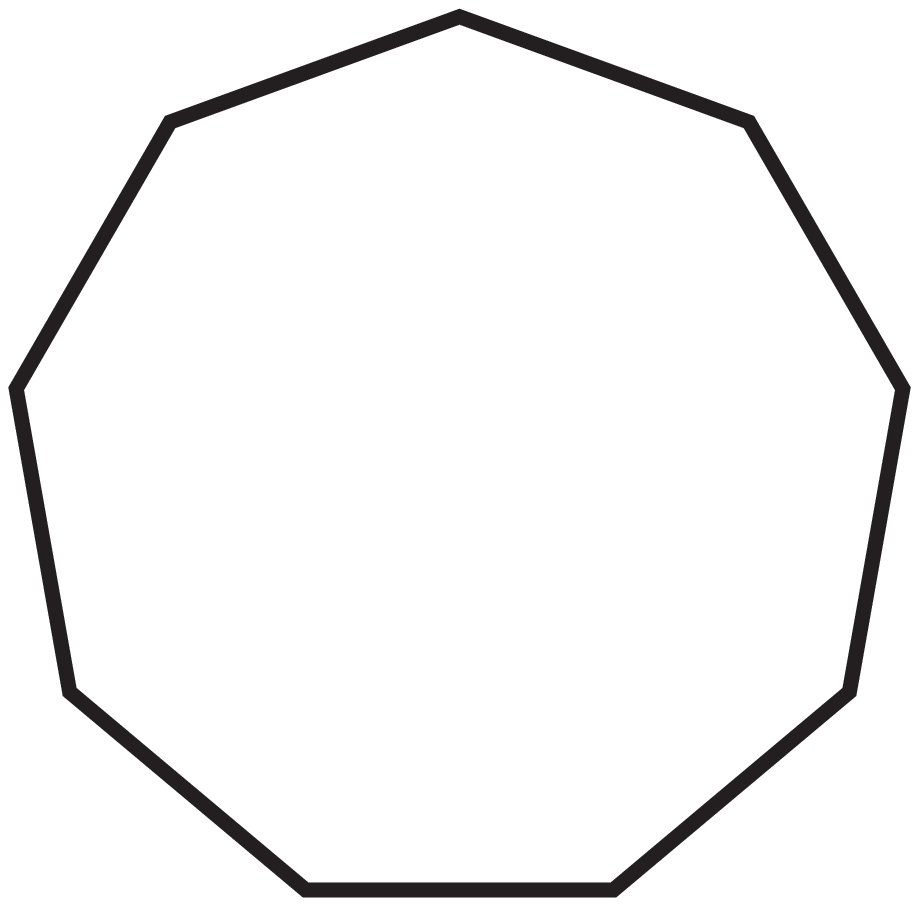}}
\hspace{0.07\textwidth}
\subfigure{\includegraphics[width=0.16\textwidth]{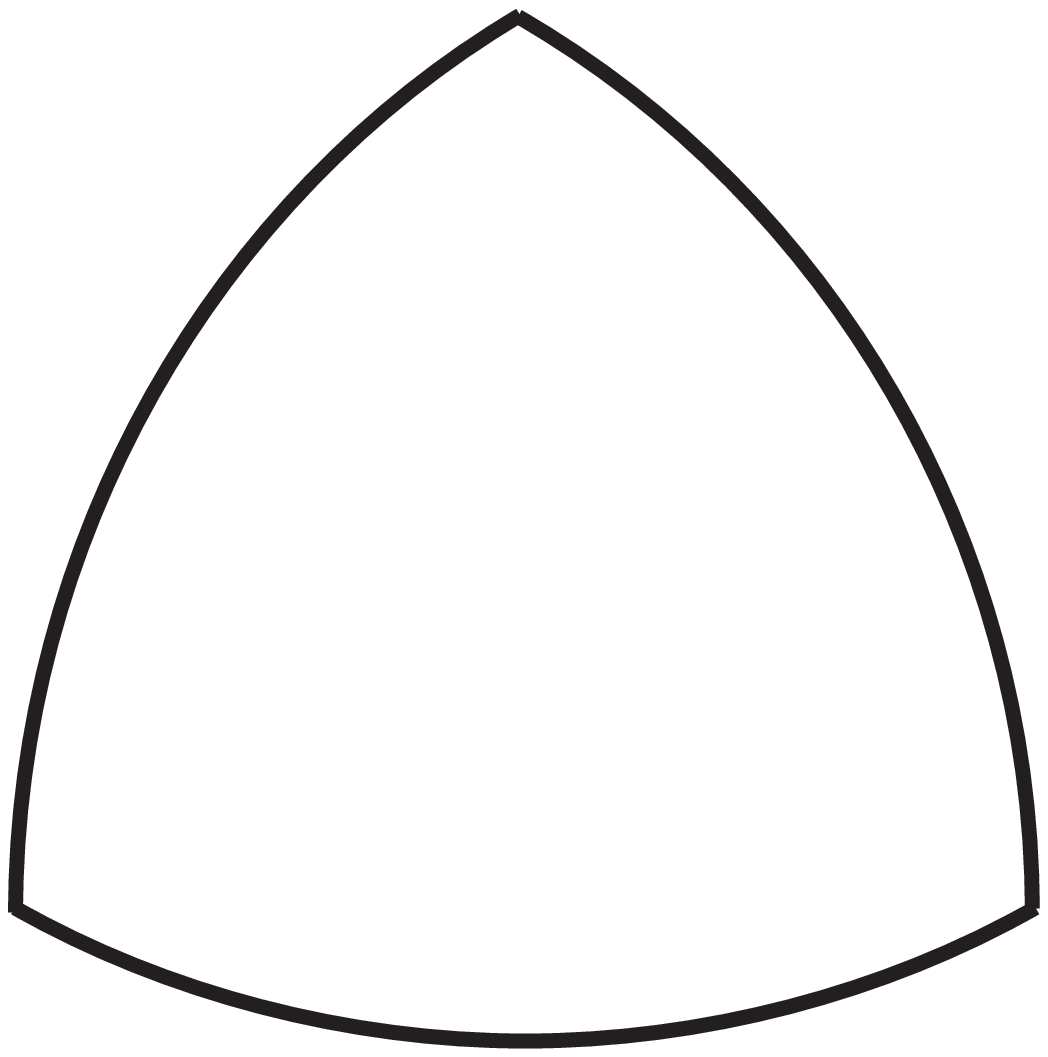}}
\caption{Some examples of 3-rotationally symmetric sets: equilateral triangle, regular enneagon, and Reuleaux triangle}
}
\end{figure}

In this work, we shall focus on the class of 3-rotationally symmetric planar convex bodies, thus assuming that
they are compact sets.
For a given set $C$ belonging to this class, we shall denote by $T_r(C)$
the smallest equilateral triangle containing $C$.
This triangle will be useful along this work, and its construction can be done as follows:
consider $m\in\ptl C$ such that $d(p,m)=d(p,\ptl C)$, where $d$ represents the Euclidean distance in the plane.
Equivalently,  $m$ is a point in $\ptl C$ achieving the minimum distance to $p$.
% and so, the segment $\overline{pm}$ will intersect $\ptl C$ orthogonally.
Consider then the line orthogonal to the segment $\overline{pm}$ passing through $m$,
and the other two lines determined by the existing threefold symmetry of $C$, see Figure~\ref{fig:tr}.
The resulting equilateral triangle is necessarily $T_r(C)$,
due to the convexity of $C$ and the fact that the apothem of that triangle is minimal by the choice of the point $m$.

% Circunscribed triangle
% como se ha considerado el punto m como el de minima distancia al centro p, se tiene que el triangulo construido es el de mínima area, al ser $\overline{pm}$ la apotema de dicho triangulo. Si hubiera un triangulo menor, su apotema también seria menor lo que contradiria que m sea el punto de minima distancia. Y por la convexidad, es claro que $C\subset T_r$.
% Si $C$ es un poligono, entonces habrá un lado de $C$ reposando sobre cada lado de $T_r$.

\begin{figure}[ht]
  \includegraphics[width=0.25\textwidth]{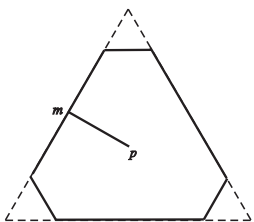}\\
  \caption{Construction of the smallest equilateral triangle
  containing a 3-rotationally symmetric hexagon}
  \label{fig:tr}
\end{figure}

% Parece que, en caso de tener un poligono 3-rotacionalmente simetrico, el triangulo T_r contendrá al mayor de los lados del poligono.

\begin{remark}
\label{re:ball}
Let $C$ be a 3-rotationally symmetric planar convex body $C$, and let $T_r(C)$ be its associated smallest
equilateral triangle. We shall denote by $B(C)$ the inscribed ball of $T_r(C)$,
whose center coincides with the center of the threefold symmetry of $C$,
and its radius is equal to the length of the apothem of $T_r(C)$. %previous segment $\overline{pm}$.
Notice that there are no points from $\ptl C$ in the interior of $B(C)$, and $B(C)\subseteq C\subseteq T_r(C)$,
see Figure~\ref{fig:ball}.
\end{remark}

\begin{figure}[ht]
  \includegraphics[width=0.25\textwidth]{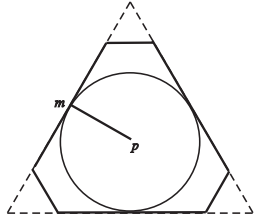}\\
  \caption{Inscribed ball associated to a 3-rotationally symmetric hexagon}
  \label{fig:ball}
\end{figure}

We now define the suitable partitions into three subsets that will be considered for the class of
3-rotationally symmetric planar convex bodies, although the definition can be done in a more general setting.

\begin{definition}
Let $C$ be a planar compact set.
A trisection $T$ of $C$ is a decomposition $\{C_1,C_2,C_3\}$ of $C$ by planar curves $\{\ga_{12}, \ga_{23}, \ga_{13}\}$ satisfying:
\begin{itemize}
\item[i)] $C=C_1\cup C_2 \cup C_3$.
\item[ii)] $\inte(C_i)\cap\inte(C_j)=\emptyset$, for $i,j\in\{1,2,3\}$, $i\neq j$.
\item[iii)] $\area(C_i)=\frac{1}{3}\area(C)$, for $i\in\{1,2,3\}$.
\item[iv)] $C_i$ is a connected set, and $\gamma_{ij}=\ptl C_i\cap\ptl C_j$, for $i,j\in\{1,2,3\}$.  %is a connected curve,
\item[v)] The curves $\ga_{12}, \ga_{23}, \ga_{13}$ meet at a point $c\in\inte(C)$,
and the other endpoints of the curves lie in $\ptl C$.
\end{itemize}
The endpoints of the curves $\{\ga_{12}, \ga_{23}, \ga_{13}\}$
meeting $\ptl C$ will be simply referred to as the endpoints of the trisection,
and the point $c\in\inte(C)$ will be called the common point of the trisection.
\end{definition}
\noindent In other words, a trisection is a partition of the original set $C$ into three connected subsets of equal areas
by three curves starting from an interior common point and ending on different points of the boundary of $C$.
It is clear that there is an infinite number of trisections for any 3-rotationally symmetric planar convex body,
since the curves are not necessarily segments and the common point $c$ does not have to coincide with the center $p$ of symmetry of the set,
as depicted in Figure~\ref{fig:tri-hex}.

\begin{figure}[ht]
  \includegraphics[width=0.77\textwidth]{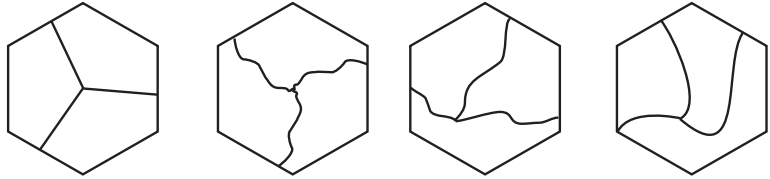}\\
  \caption{Some different trisections for a regular hexagon}
  \label{fig:tri-hex}
\end{figure}

In this work we shall consider a geometric functional defined in terms of the diameter of a set,
which is called \emph{maximum relative diameter}, which is defined as follows.

\begin{definition}
Let $C$ be a planar compact set,
and let $T$ be a trisection of $C$ into subsets $C_1,C_2,C_3$.
We define the maximum relative diameter of the trisection $T$ as
\[
d_M(T,C)=\max\{D(C_1),D(C_2),D(C_3)\},
\]
where $D(C_i)=\max\{d(x,y):x,y\in C_i\}$ denotes the diameter of $C_i$.
\end{definition}

\begin{remark}
For simplicity, we shall usually write $d_M(T)$ instead of $d_M(T,C)$, if no confusion may appear.
\end{remark}

\begin{remark}
We point out that for a given trisection of a set, the maximum relative diameter represents the
maximum distance between two points in any of the subsets of the trisection.
The existence of this value is clear due to the continuity of the Euclidean distance and
the compactness of the original set. %(which derives from the existing threefold symmetry).
\end{remark}

\begin{remark}
\label{le:minimopoligono}
We recall that the diameter of a convex polygon is attained by the distance between two of its vertices.
\end{remark}

Following lemma shows that the maximum relative diameter of a trisection is always achie\-ved
by points lying in the boundary of one of the subsets determined by the trisection.

\begin{lemma}
\label{le:borde}
%Let $C$ be a 3-rotationally symmetric planar convex body, and $T$ a trisection of $C$.
Let $C$ be a planar compact set, and $T$ a trisection of $C$.
Then, the maximum relative diameter $d_M(T)$ is attained by the distance between two points belonging
to the boundary of one of the subsets given by $T$.
\end{lemma}

\begin{proof}
Let $C_1,C_2,C_3$ be the subsets of $C$ given by the trisection $T$, and assume that $d_M(T)=d(a,b)$,
with $a,b\in C_1$. If $a\in\inte(C_1)$, we can find a small ball $B_a$ centered at $a$ and completely contained in $C_1$,
since $\inte(C_1)$ is an open set.
Then it is clear that there are points in $B_a$ whose distance from $b$ will be strictly larger than $d(a,b)$,
which contradicts the fact that $d_M(T)=d(a,b)$.
\end{proof}

As pointed out in the Introduction, the purpose of this paper is the following:
among all the trisections for a given 3-rotationally symmetric planar convex body,
which is the trisection providing the minimum possible value for the maximum relative diameter?
We shall prove that the so-called \emph{standard trisection}, described in Section~\ref{se:tb},
is an answer to this question.

\section{Standard trisection $T_b$}
\label{se:tb}

We now proceed to the description of the construction of one particular trisection for the sets of our family, %3-rotationally symmetric planar convex bodies,
called the \emph{standard trisection}.

Let $C$ be a 3-rotationally symmetric planar convex body,
and consider $T_r(C)$ its associated smallest equilateral triangle.
By joining the midpoints of the edges of $T_r(C)$ with the center $p$ of the threefold symmetry of $C$ by line segments,
we shall clearly obtain a trisection of $C$ yielding identical subsets $C_1, C_2, C_3$ up to rotation, due to the existing symmetry.
This trisection will be called \emph{standard trisection} of $C$, and will be denoted by $T_b(C)$
(or simply $T_b$ if there is no confusion with the considered set).
Notice that the standard trisection of $C$ is uniquely determined by means of the equilateral triangle $T_r(C)$.
Along this paper, we shall also denote by $\{v_1,v_2,v_3\}$ the endpoints of the standard trisection,
which coincide with the midpoints of the edges of the associated equilateral triangle.

\begin{figure}[ht]
  \includegraphics[width=0.9\textwidth]{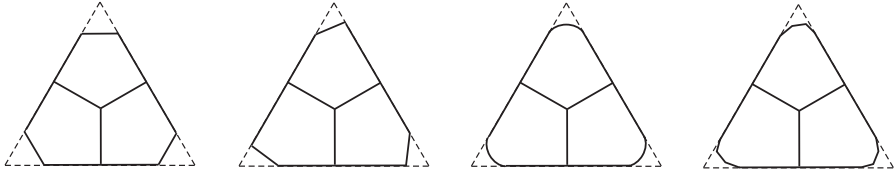}\\
  \caption{Standard trisections for different 3-rotationally symmetric sets}
  %\label{fig:tri-hex}
\end{figure}

% We shall see in Section~\ref{se:general} that this standard trisection gives the minimum value for the maximum relative diameter, among all trisections of a given 3-rotationally symmetric planar convex body.

% We establish the following conjecture on the trisection which minimizes the functional maximum relative diameter, that will be proved in the following sections.

% \begin{conjecture}
%\label{co:co}
%Let $C$ be a 3-rotationally symmetric planar convex body. Consider the standard trisection $T_b$ of $C$. Then
%\[
%d_M(T_b)\leq d_M(T),
%\]
%for any $T$ trisection of $C$.
%\end{conjecture}

Next Proposition~\ref{prop:minimodm} indicates which pair of points realizes $d_M(T_b)$ for any 3-rotationally symmetric convex \emph{polygon}. %We previously recall the following well-known lemma.

\begin{proposition}
\label{prop:minimodm}
Let $C$ be a 3-rotationally symmetric convex polygon, and $T_b$ the standard trisection of $C$. Then
\[
d_M(T_b)=d(v_1,v_2),\text{ or } d_M(T_b)=d(p,x),
\]
where $v_1$, $v_2$ are two endpoints of $T_b$, $p$ is the center of the threefold symmetry, and $x$
is a vertex of $C$.
\end{proposition}

\begin{proof}
Let $C_1, C_2, C_3$ be the identical subsets of $C$ given by the standard trisection $T_b$,
and assume that $d_M(T_b)=D(C_1)$.
Since $C_1$ is clearly a polygon, we have that $D(C_1)=d(z_1,z_2)$, for $z_1, z_2$ vertices of $C_1$,
by using Remark~\ref{le:minimopoligono}.
By construction, the vertices of $C_1$ are the center $p$ of the threefold symmetry,
the two endpoints of $T_b$ contained in $C_1$, say $v_1$, $v_2$, and the vertices of $C$ contained in $C_1$.
%Let us discard the pairs of vertices of $C_1$ that do not realize $D(C_1)$.

We have that $D(C_1)\neq d(p,v_1)$, since there are points in $\ptl C_1$ close to $v_1$
whose distance from $p$ is larger than $d(p,v_1)$. Analogously, $D(C_1)\neq d(p,v_2)$.
On the other hand, consider the equilateral triangle $\tau$ determined by $v_1$, $v_2$
and the corresponding vertex of $T_r(C)$.
The largest distance between two points in $\tau$ is given by $d(v_1,v_2)$.
Therefore, for any two arbitrary vertices $x,y$ of $C$ contained in $C_1$,
$D(C_1)$ cannot be attained either by $d(x,y)$ or by $d(v_i,x)$, with $i=1,2$,
so the remaining possibilities are $D(C_1)=d(v_1,v_2)$ or $D(C_1)=d(p,x)$.
\end{proof}

Both possibilities from Proposition~\ref{prop:minimodm} may occur, as illustrated in the following lemma.
We shall denote by $\mathscr{P}$ the family of 3-rotationally symmetric \emph{regular} polygons, which
will have necessarily $3n$ edges, $n\in\nn$.

\begin{lemma}
\label{le:min-regular}
Consider $C\in\mathscr{P}$, with $m=3n$ edges, $n\in\nn$.
Let $T_b$ be the standard trisection of $C$ with endpoints $v_1,v_2,v_3$.
Then
\begin{itemize}
\item[i)] $d_M(T_b,C)=d(v_1,v_2)$ if $n>1$, and
\item[ii)] $d_M(T_b,C)=d(p,x)$ if $n=1$,
\end{itemize}
where $p$ is the center of $C$ and $x$ is any vertex of $C$.
%$$d_M(T_b,C)=d(v_1,v_2)\text{ if }n>1,\text{ and }d_M(T_b,C)=d(p,v)\text{ if }n=1,$$
\end{lemma}

\begin{proof}
By Proposition~\ref{prop:minimodm}, we know that $d_M(T_b,C)$ is either $d(v_1,v_2)$ or $d(p,x)$.
It is straightforward checking (by using conveniently the law of sines)
that $d(v_1,v_2)=\sqrt{3}\, a(C)$, and $d(p,x)=a(C)\,\cos^{-1}(\pi/m)$,
where $a(C)$ is the apothem of $C$.
Analytically, $\sqrt{3}>\cos^{-1}(\pi/m)$ if and only if $m>3$, equivalently $n>1$,
as stated.
\end{proof}

\begin{remark}
Above lemma implies that the equilateral triangle is the \emph{unique} 3-rotationally symmetric regular polygon
with $d_M(T_b)=d(p,x)$. For the rest of polygons in $\mathscr{P}$, the maximum relative diameter for the standard
trisection is given by the distance between two of its endpoints.

\end{remark}

% El siguiente resultado y el remark asociado no es fundamental para nuestro trabajo, pero dan información sobre el diametro relativo máximo de la triseccion standard y sobre qué puntos se alcanza.

We finish this section with the following result for general 3-rotationally symmetric planar convex bodies,
which is analogous to Proposition~\ref{prop:minimodm}.

%For general 3-rotationally symmetric planar convex bodies, we have the following result, analogous to Proposition~\ref{prop:minimodm}.

\begin{proposition}
\label{prop:borde}
Let $C$ be a 3-rotationally symmetric planar convex body, and $T_b$ the standard trisection of $C$. Then
\[
d_M(T_b)=d(v_1,v_2),\text{ or } d_M(T_b)=d(p,x),
\]
where $v_1$, $v_2$ are two endpoints of $T_b$, $p$ is the center of the threefold symmetry, and $x\in\ptl C$.
\end{proposition}

\begin{proof}
Let $C_1,C_2,C_3$ be the identical subsets of $C$ given by $T_b$.
We can assume that $d_M(T_b)=d(a,b)$, with $a,b\in \ptl C_1$, in view of Lemma~\ref{le:borde}.
In fact, those points will belong to $A:=\ptl C_1\cap \ptl C$ or $B:=\ptl C_1\cap\inte(C)$.
If both of them belong to $B$, then clearly $d(a,b)<d(v_1,v_2)$, which is contradictory.
If both points belong to $A$, as $C$ is contained in the smallest triangle $T_r(C)$,
it follows that $A$ will be contained in the equilateral triangle determined by the edge $\overline{v_1v_2}$,
and so $d(a,b)\leq d(v_1,v_2)$. Finally, if $a\in A$ and $b\in B$, it is easy to check
that the only admissible possibility for $d_M(T_b)$ is $d(p,x)$, with $x\in A$.
% en este ultimo caso, se pueden discutir las distintas posibilidades para concluir que la mayor distancia vendrá dada por un punto del borde y el centro p
\end{proof}

\begin{remark}
We stress that Proposition~\ref{prop:borde} implies that one of the points realizing
the maximum relative diameter for the standard trisection will necessarily lie in the boundary of the considered set.
This property does not hold for general trisections, where that value %the maximum relative diameter
might be achieved by two points from the interior of the set.
\end{remark}

%
%We finish this section with a result describing the position of the endpoints of any minimizing trisection,
%in terms of the endpoints of the standard trisection.

\section{Trisections minimizing the maximum relative diameter}
\label{se:general}

In this section we shall prove that, for any 3-rotationally symmetric planar convex set $C$,
its associated standard trisection $T_b$ gives the minimum possible value for the maximum relative diameter $d_M$.
The key point is the previous Proposition~\ref{prop:borde}, which describes the two possible precise values for $d_M(T_b,C)$.
In fact, we shall see in Lemmata~\ref{le:primero} and \ref{le:segundo} that,
for any trisection $T$ of $C$, $d_M(T,C)$ is always greater than or equal to each of those possibilities,
leading us to our main Theorem~\ref{th:main}. Furthermore, at the end of this section
we shall discuss about the uniqueness of the minimizing trisection. Several examples will show that
uniqueness does not occur in general for this problem.

% The following two lemmata give the two announced inequalities.

\begin{lemma}
\label{le:primero}
Let $C$ be a 3-rotationally symmetric planar convex body, and let $T$ be a trisection of $C$.
For any $x\in\ptl C$, we have that $$d_M(T)\geq d(p,x),$$ where $p$ is the center of symmetry of $C$.
\end{lemma}

\begin{proof}
Let $C_1,C_2,C_3$ be the subsets of $C$ given by $T$. We can assume that $p\in C_1$.
Call $x',x''$ the two threefold symmetric points of $x$ in $\ptl C$, % which are threefold symmetric with respect to $x$.
which will satisfy $d(p,x)=d(p,x')=d(p,x'')$.
If any of these three points $x,x',x''$ belongs to $C_1$, then the thesis is trivial
since $d_M(T)\geq D(C_1)\geq d(p,x)$.
Otherwise, two of them will be necessarily contained in the same subset $C_i$, for $i\in\{2,3\}$, say $x',x''$.
Therefore $d_M(T)\geq D(C_i)\geq d(x',x'')$. It is straightforward
checking that $d(x',x'')>d(p,x)$, which finishes the proof.
\end{proof}

\begin{remark}
\label{re:uni}
In the previous proof, we obtain \emph{strict} inequality $d_M(T)>d(p,x)$ if $p\in\inte(C_1)$,
since in that case, by assuming for instance that $x\in C_1$ and proceeding as in the proof of Lemma~\ref{le:borde}, we will find points in $C_1$ whose distance to $x$ is larger than $d(p,x)$.
\end{remark}

\begin{lemma}
\label{le:segundo}
Let $C$ be a 3-rotationally symmetric planar convex body, and let $T$ be a trisection of $C$.
Then we have that $$d_M(T)\geq d(v_1,v_2),$$ where $v_1,v_2$ are endpoints of the standard trisection $T_b$ of $C$.
\end{lemma}

\begin{proof}
Call $w_1,w_2,w_3$ the endpoints of the trisection $T$, and $C_1,C_2,C_3$ the subsets of $C$ given by $T$.
We shall distinguish two cases in this proof.

\noindent \emph{Case 1:} Assume that an endpoint of $T$ coincides with an endpoint of $T_b$, say $w_1=v_1$.
Decompose $\ptl C-\{w_1\}$ into three disjoint curves $\alpha, \beta, \ga$,
where $\alpha=(w_1,v_2]$, $\beta=(w_1,v_3]$ and $\ga=(v_2,v_3)$. Clearly $w_2,w_3\in\alpha\cup\beta\cup\ga$.
It is straightforward checking, by discussing the precise placement of $w_2$ and $w_3$ along $\ptl C$,
that there will be always two endpoints $v_i,v_j$ of $T_b$ contained in a same subset $C_i$, and so
$$d_M(T)\geq D(C_i)\geq d(v_i,v_j).$$

\noindent \emph{Case 2:} Assume that the endpoints of $T$ do not coincide with the endpoints of $T_b$.
Then, all these endpoints have to be placed alternately along $\ptl C$, say $v_1,w_1,v_2,w_2,v_3,w_3$
(otherwise, two endpoints of $T_b$ will belong to a same subset $C_i$ and so $d_M(T)\geq d(v_1,v_2)$, as previously).
Rotate the standard trisection $T_b$ with its edges extended until meeting an endpoint of $T$, say $w_3$.
This rotation will move the endpoints of $T_b$ onto three new points $v_1',v_2',v_3'$ on $\ptl B(C)$,
with $d(v_i',v_j')=d(v_i,v_j)$, $i,j\in\{1,2,3\}$.
Note that $w_1$ will be contained by construction in the striped region from Figure \ref{fig:rotated},
and therefore $w_1$ is not contained in the ball centered at $w_3$ of radius $d(w_3,v_1')=d(w_3,v_2')$,
so $d(w_1,w_3)>d(w_3,v_1')$. Hence
$$d_M(T)\geq d(w_1,w_3)>d(w_3,v_1')>d(v_3',v_1')=d(v_3,v_1),$$
where the last inequality above is trivial.
\end{proof}

\begin{figure}[ht]
  \includegraphics[width=0.33\textwidth]{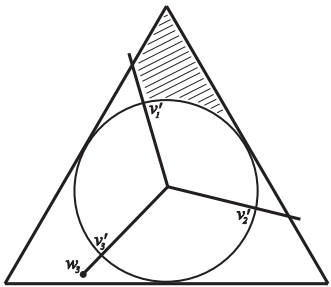}\\
  \caption{Points $v_1',v_2',v_3'$ after rotating $T_b$.
  The endpoint $w_1$ of $T$ will be contained in the striped region}
  \label{fig:rotated}
\end{figure}

\begin{remark}
\label{re:topo}
A refinement in the proof of Lemma~\ref{le:segundo} gives the \emph{strict} inequality $d_M(T)>d(v_1,v_2)$,
unless the three endpoints of $T$ coincide with the endpoints of the standard trisection $T_b$.
In fact, if the endpoint $v_1$ of $T_b$ is not an endpoint of $T$,
we can find a point $q$ close to $v_1$, and another endpoint $v_i$ of $T_b$,
with $q,v_i$ contained in a same subset given by $T$, and satisfying $d(q,v_i)>d(v_1,v_i)$.
\end{remark}

% Escribimos este remark sin detallar en exceso porque la demostración formal implica cierta distinción exhaustiva de casos, que lo haria demasiado largo. En todos los casos, siempre es posible encontrar un endpoint de T_b y un punto q próximo a otro endpoint v_1 de T_b (el que no coincide con endpoints de T) de forma que ambos están en el mismo subconjunto C_i, y con d(q,v_i)>d(v_1,v_i). Asi d_M(T)\geq D(C_i)\geq d(q,v_i)>d(v_1,v_i).

The two previous lemmata allow us to state our main result: for any 3-rotationally symmetric
planar convex body, its associated standard trisection gives the minimum value for the maximum relative diameter.

\begin{theorem}
\label{th:main}
Let $C$ be a 3-rotationally symmetric planar convex body, and $T_b$ its associated standard trisection. Then,
\[
d_M(T_b,C)\leq d_M(T,C),
\]
for any trisection $T$ of C.
\end{theorem}

\begin{proof}
From Proposition~\ref{prop:borde} we have that $d_M(T_b,C)$ equals either $d(v_1,v_2)$ or $d(p,x)$,
where $v_1,v_2$ are endpoints of $T_b$, $p$ is the center of symmetry and $x\in\ptl C$.
By using Lemmata~\ref{le:primero} or \ref{le:segundo}, we conclude that $d_M(T,C)\geq d_M(T_b,C)$.
\end{proof}

%\begin{remark}
%We point out that the \emph{maximum} value for the maximum relative diameter is trivially given by any trisection $T$
%with $d_M(T)$ equal to the diameter of the consider set.
%\end{remark}

\subsection{Uniqueness of the minimizing trisection}

For this particular problem involving the maximum relative diameter functional,
uniqueness of the minimizing trisection does not hold.
%This fact was already pointed out in \cite{mps} when studying minimizing bisections for centrally symmetric planar convex sets.
This can be observed with some simple examples.
Figure~\ref{fig:unicidad} shows some trisections of the regular hexagon $H$
obtained by small-enough appropriate deformations of its standard trisection $T_b(H)$,
which give the same minimum value for our functional.
And Figure~\ref{fig:unicidad2} shows that, by rotating slightly the standard trisection for the equilateral triangle $\mathcal{T}$,
we obtain new trisections with the same maximum relative diameter.
Thus, the minimizing trisection for a set of our class is not unique, and we cannot aim to find a complete characterization of \emph{all} of them.
% no es posible caracterizar todas las trisecciones minimizantes: hay infinitas.
% para el triangulo, ni siquiera hay unicidad en el caso de trisections by line segments.

\begin{figure}[ht]
  \includegraphics[width=0.8\textwidth]{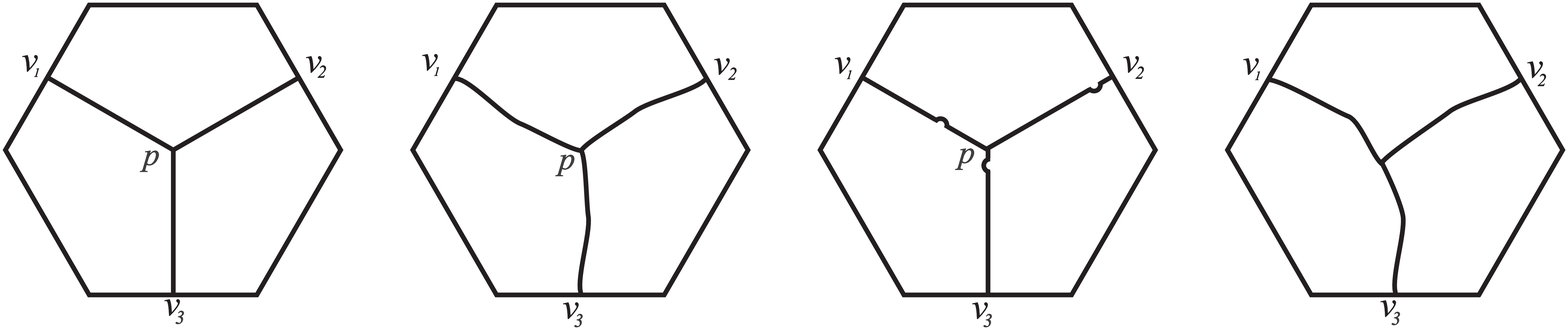}\\
  \caption{Four different minimizing trisections for $H$.
  The first trisection is $T_b(H)$, and the rest of them are obtained by small deformations.
  In the last one, the common point does not coincide with the center of symmetry}
  \label{fig:unicidad}
\end{figure}

\begin{figure}[ht]
  \includegraphics[width=0.69\textwidth]{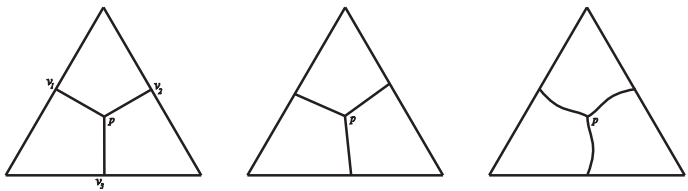}\\
  \caption{Three different minimizing trisections for $\mathcal{T}$:
  the standard trisection $T_b(\mathcal{T})$, a trisection obtained by rotating slightly $T_b(\mathcal{T})$,
  and a trisection determined by curves}
  \label{fig:unicidad2}
\end{figure}

%Remarks~\ref{re:uni} and \ref{re:topo} give necessary conditions
%for a trisection to be minimizing. More precisely, taking into account Theorem~\ref{th:main}, Proposition~\ref{prop:borde} and
%the notation therein, if the maximum relative diameter for the standard trisection $d_M(T_b)$ of a set is given
%by $d(v_i,v_j)$, then the endpoints of any minimizing trisection must coincide with the endpoints
%of the standard one. On the other hand, if $d_M(T_b)$ is equal to $d(p,x)$, the common point of
%any minimizing trisection will necessarily be the center of symmetry of the set.

Remarks~\ref{re:uni} and \ref{re:topo} provide necessary conditions for a trisection to be minimizing.
More precisely, for a 3-rotationally symmetric planar convex body $C$, and any minimizing trisection $T$,
it follows that $d_M(T)$ will be equal to $d_M(T_b)$,
where $T_b$ is the standard trisection associated to $C$, by using Theorem~\ref{th:main}.
Taking into account Proposition~\ref{prop:borde} and the notation therein,
if $d_M(T_b)$ is given by $d(v_i,v_j)$, then the endpoints of $T$ will coincide with the endpoints of $T_b$ due to Remark~\ref{re:topo}. On the other hand, if $d_M(T_b)$ is given by $d(p,x)$, Remark~\ref{re:uni} implies that
the common point of $T$ will necessarily be the center of symmetry of the set.

Finally, if we restrict our problem to trisections \emph{by line segments},
we can assure uniqueness only if the maximum relative diameter for the standard trisection
is given by $d(v_i,v_j)$, due to Remark~\ref{re:topo} and the fact that, in that case, the
common point of the trisection must coincide with the center of symmetry of the set, in order to
satisfy the subdivision area condition.
%In this case, uniqueness follows from Remark~\ref{re:topo}, and the fact that the trisection must have the center of symmetry of the set as common point, in order to satisfy the area condition.
In the other case, when the maximum relative diameter for the standard trisection equals $d(p,x)$,
Figure~\ref{fig:unicidad2} shows that uniqueness does not occur even for trisections by line segments.

\section{Lower bound for $d_M$ in the class of 3-rotationally symmetric planar convex bodies}
\label{se:min}

In this section we shall discuss the following related questions: for a given 3-rotationally symmetric planar convex body $C$
of unit area, and an arbitrary trisection $T$ of $C$, can we estimate a lower bound for $d_M(T,C)$?
%does a lower bound for $d_M(T,C)$ exist?
%can we give a lower bound for $d_M(T,C)$?
Can we characterize a set attaining such lower bound?
These questions are interesting in the context of fencing problems, because
they lead to useful relative geometrical inequalities \cite{pe}.
We remark that this problem is analogous to the one studied in \cite{mps},
where a lower bound for the maximum relative diameter of \emph{bisections} of \emph{centrally symmetric} planar convex
bodies is obtained, describing also the corresponding optimal body \cite[Th.~5]{mps}.

In our setting, it is clear that we have to focus on the standard trisections,
which give the minimum value for the maximum relative diameter functional $d_M$ for any set in our family, in view of Theorem~\ref{th:main}.
For homogeneity, we shall consider the class of 3-rotationally symmetric planar convex bodies of \emph{unit area},
denoted by $\mathscr{C}_1$. Hence our problem consists of estimating
\[
\min\{d_M(T_b,C): C\in\mathscr{C}_1\}.
\]

\noindent We shall start with the subfamily $\mathscr{P}_1$ of $\mathscr{C}_1$
composed by the regular 3-rotationally symmetric polygons of unit area.
Recall that these polygons will have necessarily $m=3n$ edges, with $n\in\nn$.
Taking into account Lemma~\ref{le:min-regular}, where the maximum relative diameter is precisely computed for the
sets in $\mathscr{P}_1$, we shall prove that the minimum value is given by the equilateral triangle.
% for which Lemma~\ref{le:minimopoligono} and Proposition~\ref{prop:minimodm} enable to obtain the desired minimum value.

%\begin{lemma}
%%\label{le:min-regular}
%Consider $C\in\mathscr{P}$, with $m=3n$ edges, $n\in\nn$.
%Let $T_b$ be the standard trisection of $C$ with endpoints $v_1,v_2,v_3$.
%Then
%\begin{itemize}
%\item[i)] $d_M(T_b,C)=d(v_1,v_2)$ if $n>1$, and
%\item[ii)] $d_M(T_b,C)=d(p,x)$ if $n=1$,
%\end{itemize}
%where $p$ is the center of $C$ and $x$ is any vertex of $C$.
%%$$d_M(T_b,C)=d(v_1,v_2)\text{ if }n>1,\text{ and }d_M(T_b,C)=d(p,v)\text{ if }n=1,$$
%\end{lemma}
%
%\begin{proof}
%By Proposition~\ref{prop:minimodm}, we know that $d_M(T_b,C)$ is either $d(v_1,v_2)$ or $d(p,x)$.
%It is straightforward checking (by using conveniently the law of sines)
%that $d(v_1,v_2)=\sqrt{3}\, a(C)$, and that $d(p,x)=a(C)\,\cos^{-1}(\pi/m)$,
%where $a(C)$ is the apothem of the polygon.
%Analytically, $\sqrt{3}>\cos^{-1}(\pi/m)$ if and only if $m>3$, equivalently $n>1$,
%as stated.
%\end{proof}

\begin{theorem}
\label{th:min}
Let $\mathscr{P}_1\subset\mathscr{C}_1$ be the family of regular 3-rotationally symmetric polygons of unit area.
Then, $$\min\{d_M(T_b,C):C\in\mathscr{P}_1\}=d_M(T_b,\mathcal{T}),$$ where $\mathcal{T}$ is the equilateral triangle.
\end{theorem}

\begin{proof}
For any polygon $C\in\mathscr{P}_1$ with $m=3n$ edges, $n\in\nn$,
it is easy to check that the apothem $a(C)$ of $C$ is equal to $m^{-1/2}\,\cot^{1/2}(\pi/m)$.
Using this fact, we can compare the maximum relative diameter of the sets in $\mathscr{P}_1$,
taking into account Lemma~\ref{le:min-regular} and its proof.

From Lemma~\ref{le:min-regular}, if $n>1$, then
$$d_M(T_b,C)=d(v_1,v_2)=\sqrt{3}\,a(C)=\sqrt{3}\,m^{-1/2}\,\cot^{1/2}(\pi/m),$$
where $v_1,v_2$ are two endpoints of $T_b$.
It is straightforward checking that the above expression is increasing with respect to $m$,
and so the minimum value will be attained when $m=6$, equivalently $n=2$
(which is the case of the regular hexagon $H$), with $d_M(T_b,H)$ equal to $2^{-1/2}\,\cot^{1/2}(\pi/6)$.
On the other hand, if $n=1$, $C$ is an equilateral triangle $\mathcal{T}$
and $$d_M(T_b,\mathcal{T})=d(p,x)=a(C)\,\cos^{-1}(\pi/3)=2\,\cot^{1/2}(\pi/3)\,3^{-1/2},$$
where $p$ is the center of symmetry and $x$ is any vertex of $\mathcal{T}$.
Since $d_M(T_b,\mathcal{T})<d_M(T_b,H)$, the result follows.
\end{proof}

\begin{remark}
We point out that Theorem~\ref{th:min} provides a lower bound for the maximum relative
diameter of any trisection of a regular polygon contained in $\mathscr{P}_1$. That is,
for $C\in\mathscr{P}_1$,
$$d_M(T,C)\geq d_M(T_b,\mathcal{T})=2\,\cot^{1/2}(\pi/3)\,3^{-1/2}=0.877383\dots,$$
for any trisection $T$ of $C$.
\end{remark}

In \cite[Th.~5]{mps} it is proved that the minimum value of the maximum relative diameter
for bisections, in the class of centrally symmetric planar convex bodies, is attained by
a particular set delimited by the intersection of a strip and two symmetric circular sectors meeting in a certain angle.
In some sense, this result may suggest that the minimizing set in our setting could have curved edges.
Since Reuleaux triangle is a remarkable set of our class with that property,
we shall proceed to the study of this candidate in detail, showing in the following examples that
it does not actually provide the desired minimum value. %
%
%This result could suggest that the Reuleaux triangle is a good candidate
%for minimizing the maximum relative diameter for trisections,
%since it is a sort of analogous set in our family.
%However, the following examples show that this is not the case.

\begin{example}
Consider a Reuleaux triangle $R$ of unit area. This implies that the equilateral triangle
obtained by joining the vertices of $R$ has edges of length $a=(2/(\pi-\sqrt{3}))^{1/2}$.
In view of Theorem~\ref{th:main}, the standard trisection $T_b$ gives the minimum value for $d_M$,
and taking into account Proposition~\ref{prop:borde}, straightforward computations yield that $d_M(T_b,R)=d(v_1,v_2)=(\sqrt{3}-1)\,a=0.872002\dots$, where $v_1,v_2$ are two endpoints of $T_b$.
Notice that $d_M(T_b,R)<d_M(T_b,\mathcal{T})$, where $\mathcal{T}$ is the equilateral triangle of unit area.
However, we shall see in Example~\ref{ex:ht} that the minimum value for
the maximum relative diameter is strictly smaller than $d_M(T_b,R)$.
\end{example}

\begin{example}
\label{ex:ht}
% We shall now study some convex 3-rotationally non-regular polygons, which can be constructed from a regular hexagon.

Consider a regular hexagon $H$ with edges of length $h>0$, and the smallest equilateral triangle $T_r(H)$ containing $H$.
Enlarge each edge of $H$ lying on $\ptl T_r(H)$ by adding symmetrically two segments of length $0<\varepsilon\leq h$.
By joining the endpoints of these new edges, we shall obtain a 3-rotationally symmetric convex non-regular hexagon $H_\varepsilon$. Call $a,b$ the lengths of the edges of $H_\varepsilon$, with $a<b$.
% b=lado del hexágono inicial más $2\varepsilon$

\begin{figure}[ht]
  \includegraphics[width=0.26\textwidth]{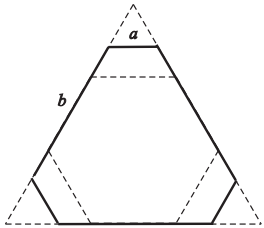}\\
  \caption{A set $H_\varepsilon$ with edges of lengths $a$, $b$}
  %\label{}
\end{figure}

By imposing that these sets enclose unit area,
we have that $H_\varepsilon$ evolve continuously, when $\varepsilon$ increases,
from the regular hexagon (with $a=b=2^{1/2}3^{-3/4}$) to the equilateral triangle (with $a=0$ and $b=2\ 3^{-1/4}$).
Morever, taking into account that each $H_\varepsilon$ can be seen as an equilateral triangle with edges of length $b+2a$
where three equilateral triangles with edges of length $a$ have been removed from the corners, the unit area condition satisfied by $H_\varepsilon$ can be expressed analytically in terms of $a,b$, obtaining
\[
b=-2\,a+\sqrt{3^{-1/2}\,4+3\,a^2},
%b=\frac{1}{3}\bigg(-6a+\sqrt{12\sqrt{3}+27a^2}\,\bigg),
\]
where $a\in[0,2^{1/2}3^{-3/4}]$.

From Theorem~\ref{th:main}, the minimum value for $d_M$
will be given by the standard trisection $T_b(H_\varepsilon)$.
By using Proposition~\ref{prop:minimodm} we have that
$d_M(T_b,H_\varepsilon)$ equals either $d(p,v)$ or $d(v_1,v_2)$, where $p$ is the center of the threefold symmetry,
$x$ is a vertex of $H_\varepsilon$ and $v_1,v_2$ are two endpoints of $T_b$.
% esto ya ha sido comprobado explicitamente con los correspondientes cálculos en unas notas previas.
In fact, that minimum value will depend on the \emph{proximity} of $H_\varepsilon$
to an equilateral triangle or a regular hexagon.
In this setting, it can be checked that
\begin{align*}
&d(p,x)=\frac{1}{3}\bigg(4\sqrt{3}+18a^2-3a\sqrt{12\sqrt{3}+27a^2}\bigg)^{1/2},
\\
&d(v_1,v_2)=\frac{1}{2}\bigg(3a^2+4/\sqrt{3}\bigg)^{1/2}.
\end{align*}

Figure~\ref{gra} shows the graphs of the above expressions as functions of $a$,
where the decreasing one corresponds to $d(p,x)$.
By the definition of maximum relative diameter functional, we have that $d_M(T_b,H_\varepsilon)$
will be the larger value between $d(p,x)$ and $d(v_1,v_2)$.
As both graphs intersect when $a=a_0=0.141227\dots$, it follows that
$d_M(T_b,H_\varepsilon)$ equals $d(p,x)$ if $a\leq a_0$ (when $H_\varepsilon$ is closer to an equilateral triangle),
and equals $d(v_1,v_2)$ if $a\geq a_0$ (when $H_\varepsilon$ is, in some sense, more similar to a regular hexagon).

\end{example}

\begin{figure}[ht]
    \includegraphics[width=0.6\textwidth]{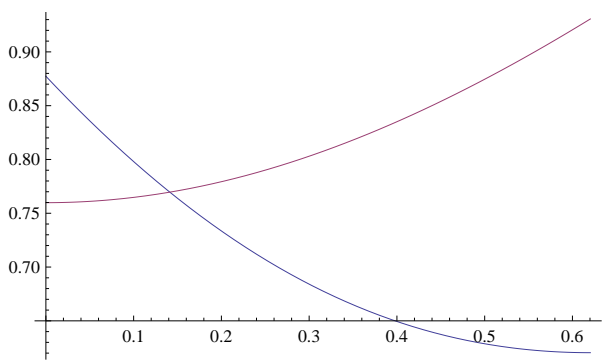}\\
  \caption{The decreasing graph corresponds to $d(p,x)$, and the increasing one
  corresponds to $d(v_1,v_2)$}
  \label{gra}
\end{figure}

In any case, it is clear that there are values of $a$, and thus sets of type $H_\varepsilon$,
such that $d_M(T_b,H_\varepsilon)$ is strictly smaller than $d_M(T_b,R)=0.872002\dots$, and so the minimum
value for the maximum relative diameter is not achieved by Reuleaux triangle $R$.
More precisely, the minimum value for $d_M$ for sets of type $H_\varepsilon$ is equal to $0.769616\dots$,
being attained when $a=a_0$.

\begin{remark}
We point out that each set $H_\varepsilon$ from Example~\ref{ex:ht} coincides with a proper Minkowski sum $\lambda\mathcal{T}+(1-\lambda) H$, where $\mathcal{T}$ is an equilateral triangle, $H$ is a regular hexagon, and $\lambda\in[0,1]$, after scaling in order to enclose unit area.
\end{remark}

The previous example suggests us the following idea, which will lead us to the desired optimal set for our problem.

\begin{example}
\label{ex:optimo}
Consider the set $H_\varepsilon$ with $a=a_0$ and its
associated smallest triangle $T_r(H_\varepsilon)$, being $p$ its center of symmetry.
Set $r:=d(p,x)$, where $x$ is any vertex of $H_\varepsilon$,
and substitute the edges of $H_\varepsilon$ of length $a$ by circular arcs of radius $r$ centered at $p$,
see Figure~\ref{fig:rounding}.
By shrinking appropriately in order to enclose unit area, we obtain a new 3-rotationally symmetric convex body that will be called $\widetilde{H}$.

\begin{figure}[ht]
  \includegraphics[width=0.56\textwidth]{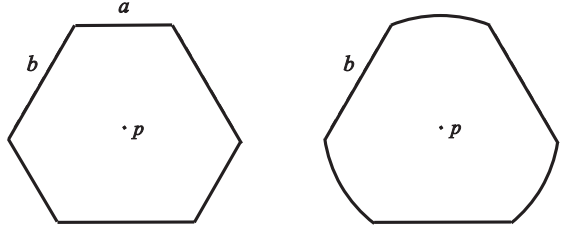}\\
  \caption{$\widetilde{H}$ is obtained by rounding some edges in the optimal set $H_\varepsilon$,
  and shrinking for unit area}
  \label{fig:rounding}
\end{figure}

Notice that, since we have to shrink in order to preserve unit area,
all the distances will decrease and so the maximum relative diameter for $\widetilde{H}$
(which by construction will be given by $d(v_i,v_j)=d(p,x)$, being $x$ any point in the circular arcs)
will be \emph{strictly less} than the corresponding value for the optimal set of type $H_\varepsilon$.
Next result shows that this set $\widetilde{H}$ is the only one attaining the minimum possible value
for the maximum relative diameter in the class of unit area 3-rotationally symmetric planar convex bodies.
\end{example}

\begin{theorem}
\label{th:optimo}
Let $C$ be a 3-rotationally symmetric planar convex body of unit area, and let $\widetilde{H}$ be
the set described in Example~\ref{ex:optimo}. Then,
\[
d_M(T_b,\widetilde{H})\leq d_M(T_b,C),
\]
with equality if and only if $\widetilde{H}=C$.
%where $\widetilde{H}$ is the set described previously.
\end{theorem}

\begin{proof}
Consider $T_r(\wtilde{H})$ and $T_r(C)$ the smallest equilateral triangles
containing $\wtilde{H}$ and $C$, respectively.
Due to the threefold symmetry, we can assume that both triangles are centered at the same point $p$,
and so the corresponding standard trisections are given by the same lines leaving from $p$.
%On the other hand, as $C$ and $\widetilde{H}$ are sets of unit area, it is not possible neither $C\subset\widetilde{H}$ nor $\wtilde{H}\subset C$. Thus there are points from $C$ which are not contained in $\widetilde{H}$.

If $T_r(\wtilde{H})\subset T_r(C)$, then clearly the distance between two endpoints of $T_b(C)$ is strictly larger than
the distance between two endpoints of $T_b(\widetilde{H})$, and so $d_M(T_b,C)>d_M(T_b,\widetilde{H})$.

Assume now the other possibility $T_r(C)\subseteq T_r(\wtilde{H})$,
taking into account that both triangles are equilateral and centered at $p$.
Then $C\subset T_r(C)\subseteq T_r(\wtilde{H})$.
Notice that it is not possible that $C\subset\wtilde{H}$, since both sets enclose unit area.
Hence, we can find a point $y\in\ptl C\subset C$ which is not contained in $\wtilde{H}$ but included in $T_r(C)$ (in fact, such a point will be placed nearby the corners of $T_r(C)$, see Figure~\ref{fig:proof}).
Then $d(p,y)>d(p,x)$, for any $x$ in the circular arcs of $\wtilde{H}$, which implies $d_M(T_b,C)>d_M(T_b,\wtilde{H})$.
\end{proof}

%\begin{proof}
%Let $T_r(\wtilde{H})$ and $T_r(C)$ be the smallest equilateral triangles
%containing $\wtilde{H}$ and $C$, respectively.
%Due to the threefold symmetry, we can assume that both triangles are centered at the same point $p$,
%and so the standard trisections associated to both sets are given by the same lines leaving from $p$.
%
%If $T_r(\wtilde{H})\subseteq T_r(C)$, the endpoints $v_1$, $v_2$ of $T_b(\wtilde{H})$ will be clearly contained in one of the subsets of $C$ generated by $T_b(C)$, and so $d(v_1,v_2)\leq d_M(C,T_b)$. Consequently, $d_M(T_b,\widetilde{H})\leq d_M(T_b,C)$.
%
%Assume now the other possibility $T_r(C)\subset T_r(\wtilde{H})$, taking into account that both triangles are equilateral and centered at $p$. Then $C\subset T_r(C)\subset T_r(\wtilde{H})$.
%Notice that it is not possible that $C\subset\wtilde{H}$, since both sets enclose unit area.
%Hence, it is easy to check that we can find a point $y\in\ptl C\subset C$ which is not contained in $\wtilde{H}$ but included in $T_r(C)$ (in fact, such a point will be placed nearby the corners of $T_r(C)$, see Figure~\ref{fig:proof}).
%Then $d(p,y)>d(p,x)$, for any $x$ in the circular arcs of $\wtilde{H}$, which clearly implies that $d_M(T_b,C)>d_M(T_b,\wtilde{H})$.
%\end{proof}

\vspace{-1mm}
\begin{figure}[ht]
  \includegraphics[width=0.3\textwidth]{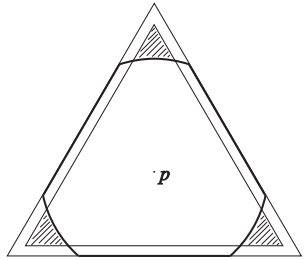}\\
  \caption{$\widetilde{H}$ and the two nested equilateral triangles $T_r(\widetilde{H})$ and $T_r(C)$. The striped regions will necessarily contain points from $C$}
  \label{fig:proof}
\end{figure}

\begin{remark}
\label{re:quotient}
We point out that the question treated in this Section~\ref{se:min}
of finding the minimum possible value for the maximum relative diameter functional
on the class $\mathscr{C}_1$ of unit area 3-rotationally symmetric planar convex bodies,
is equivalent to minimizing the quotient
\begin{equation}
\label{eq:quotient}
\frac{d_M(T_b,C)\,^2}{\area(C)},
\end{equation}
among all 3-rotationally symmetric planar convex bodies.
Notice that \eqref{eq:quotient} is invariant under dilations,
which allows to omit the unit area restriction.
This approach has been used in several works studying analogous problems \cite{cms,css,mps}.
In our setting, Theorem~\ref{th:optimo} assures that
\begin{equation}
\label{eq:bound}
\frac{d_M(T_b,C)\,^2}{\area(C)}\geq \frac{d_M(T_b,\widetilde{H})\,^2}{\area(\widetilde{H})},%=0.591764\dots,
\end{equation}
for any 3-rotationally symmetric planar convex body $C$. In fact, straightforward computations will give that $d_M(T_b,\widetilde{H})$ equals $0.769262\dots$, a value slightly less than $d_M(T_b,H_\varepsilon)$ when $a=a_0$,
and so the optimal bound in \eqref{eq:bound} is $0.591764\dots$.
\end{remark}

\begin{remark}
We stress that the optimal set $\widetilde{H}$ can be geometrically constructed by considering the intersection of an equilateral triangle and an appropriate circle with the same center, imposing unit area. This set has already appeared in literature for some optimization problems. More precisely, it is the solution of some complete systems of inequalities for 3-rotationally symmetric planar convex bodies, involving classical geometric magnitudes, see \cite[\S.~4]{hss}.
\end{remark}

\begin{figure}[ht]
  \includegraphics[width=0.3\textwidth]{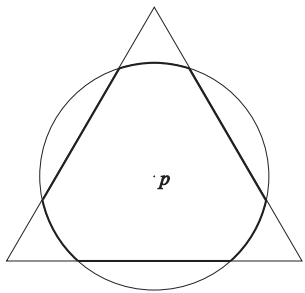}\\
  \caption{The set $\widetilde{H}$ is given by a unit area intersection of an equilateral triangle and a circle}
  %\label{}
\end{figure}

\section{General subdivisions}
\label{se:final}

For a 3-rotationally symmetric planar convex body $C$, one can consider our problem of minimizing the maximum relative diameter functional not only for trisections
(given by three curves leaving from an interior point and meeting the boundary of $C$), but for general \emph{partitions}
of $C$ into three connected subsets of equal areas.

\begin{figure}[ht]
  \includegraphics[width=0.68\textwidth]{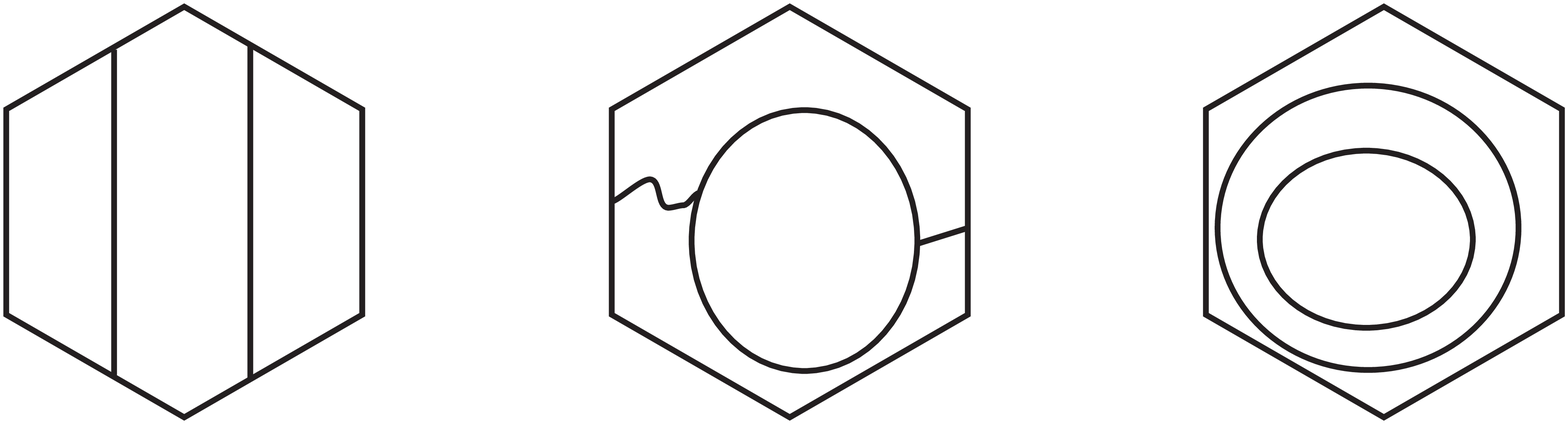}\\
  \caption{Three different partitions of the regular hexagon}
\end{figure}

In this setting, we have the following result stating that the minimum value for the maximum relative diameter
when considering general partitions is given by the corresponding standard trisection.

\begin{proposition}
\label{prop:general}
Given a 3-rotationally symmetric planar convex body $C$,
and a partition $Q$ of $C$ into three connected subsets of equal areas, we have that
$$d_M(T_b,C)\leq d_M(Q,C),$$
where $T_b$ is the standard trisection of $C$.
\end{proposition}

\begin{proof}
By Proposition~\ref{prop:borde} we know that $d_M(T_b,C)$ equals either $d(p,x)$ or $d(v_1,v_2)$,
where $p$ is the center of symmetry, $x\in\ptl C$, and $v_1,v_2$ are two endpoints of $T_b$.
Call $C_1,C_2,C_3$ the subsets of $C$ given by the partition $Q$.

If $d_M(T_b,C)=d(p,x)$, call $x',x''$ the threefold symmetric points of $x$, and let $C_1$ be a subset containing $p$.
%Fix $C_1$ as a subset containing $p$ (if it is contained in more than one subset, we just fix one of them).
In case that $C_1$ contains $x$ or one of its symmetric points, then $d_M(Q,C)\geq D(C_1)\geq d(p,x)=d_M(T_b,C)$.
Otherwise, if $C_1$ does not contain any of those points, then clearly two of them will be contained
in a same subset, say $C_2$, and so $d_M(Q,C)\geq D(C_2)\geq d(x,x')$. Since $d(x,x')>d(p,x)$,
we conclude that $d_M(Q,C)<d_M(T_b,C)$, as desired.

On the other hand, if $d_M(T_b,C)=d(v_1,v_2)$, let $q$ be the number of intersection points
between $\ptl C$ and the curves of $Q$.
% we shall discuss the different possibilities depending on the number of intersections between the curves of $Q$ and $\ptl C$.
In case that $q$ equals zero or one, it is clear that $\ptl C$ (and in particular the endpoints of $T_b$)
will be contained in one of the subsets $C_j$, and so $d_M(Q,C)\geq D(C_j)\geq d(v_1,v_2)=d_M(T_b,C)$.
If $q=2$, then $\ptl C$ will be contained in the union of two subsets $C_i\cup C_j$,
and so two endpoints of $T_b$ will necessarily lie in one subset, say $v_1,v_2\in C_i$.
Then $d_M(Q,C)\geq D(C_i)\geq d(v_1,v_2)=d_M(T_b,C)$.
If $q=3$, then $Q$ will be either a trisection, and so the statement is true by Theorem~\ref{th:main},
or two curves of $Q$ meet at a same point $y$ of $\ptl C$. In this latter case, a subset $C_i$ will
contain the two antipodal points $y,\overline{y}$ of $\ptl C$
(that is, the segment joining $y$ and $\overline{y}$ passes through the center of symmetry of $C$),
and then $d_M(Q,C)\geq D(C_i)\geq d(y,\overline{y})\geq d(v_1,v_2)=d_M(T_b,C)$, by using Lemma~\ref{le:antipodal}.
If $q=4$, then $Q$ will consist of two curves meeting $\ptl C$.
In this situation, it is easy to check that one of the subsets $C_i$
will contain two antipodal points, and proceeding as in the case $q=3$ it follows that $d_M(Q,C)\geq d_M(T_b,C)$.
Since any partition with more than four intersection points divides $C$ into more than three subsets,
we conclude that the statement holds.
%If we had more than four intersection points, then the partition would give more than three subsets, which is out of our study. Then the statement holds if $C$ is a polygon.
\end{proof}

\begin{remark}
Taking into account that any trisection is, in particular, a partition, Proposition~\ref{prop:general} implies that the minimum value for the maximum relative diameter when considering general partitions is attained by the standard trisection.
% y entonces, en este contexto, la mejor forma de dividir nuestros conjuntos es por medio de trisecciones, sin que tengan interes otro tipo de particiones.
\end{remark}

\begin{lemma}
\label{le:antipodal}
Let $C$ be a 3-rotationally symmetric planar convex set, with $p$ its center of symmetry.
Let $y,\overline{y}\in\ptl C$ be two antipodal points (that is, the segment joining $y$ and $\overline{y}$ passes through $p$). Then, $d(y,\overline{y})\geq d(v_1,v_2)$, where $v_1,v_2$ are two endpoints of the standard trisection of $C$.
\end{lemma}

\begin{proof}
Recall that $d(p,v_i)=d(p,\ptl C)=\min\{d(p,x):x\in\ptl C\}$, $i=1,2,3$, by construction of the standard trisection of $C$.
Then $d(p,y)\geq d(p,v_1)$ and $d(p,\overline{y})\geq d(p,v_1)$, and so
$$d(y,\overline{y})=d(p,y)+d(p,\overline{y})\geq d(p,v_1)+d(p,v_2)\geq d(v_1,v_2),$$
where we have used that $y,\overline{y}$ are antipodal points in the first equality, and the classical
triangular inequality in the last inequality.
\end{proof}

\end{document}